\documentclass{amsart}
\usepackage{amsmath, amssymb, amsthm, epsfig, graphicx, color, booktabs, url, a4wide}

\let\oldmarginpar\marginpar
\renewcommand\marginpar[1]{\-\oldmarginpar[\raggedleft\footnotesize #1]%
{\raggedright\footnotesize #1}}

\usepackage[colorlinks]{hyperref}
\hypersetup{linkcolor=blue, urlcolor=blue, citecolor=red}

\newtheorem{theorem}{Theorem}[section]
\newtheorem{corollary}[theorem]{Corollary}
\newtheorem{lemma}[theorem]{Lemma}

\newtheorem{definition}[theorem]{Definition}
\newenvironment{defn}{\begin{definition}\rm}{\end{definition}}

\newcommand{\proofref}[1]{\proof[Proof of Theorem \ref{#1}.]}

\renewcommand{\to}{\longrightarrow}
\newcommand{\lcm}{\operatorname{lcm}}
\newcommand{\im}{\operatorname{im}}
\newcommand{\N}{\mathbb{N}}

\newcommand{\U}{H}
\newcommand{\id}{\operatorname{id}}

\newcommand{\set}[2]{\ensuremath{\{\: #1 \: |\: #2 \:\}}}

\title{The number of nilpotent semigroups of degree 3}
\author{Andreas Distler and James D. Mitchell}
\thanks{The first author acknowledges the supported by 
  the University of St Andrews and the project PTDC/MAT/101993/2008 of
  Centro de \'Algebra da Universidade de Lisboa, financed by FCT and
  FEDER.}

\begin{document}

\begin{abstract}
A semigroup is \emph{nilpotent} of degree $3$ if it has a zero, every
product of $3$ elements equals the zero, and some product of $2$
elements is non-zero. It is part of the folklore of semigroup theory
that almost all finite semigroups are nilpotent of degree $3$. 

We give formulae for the number of nilpotent semigroups of degree $3$
with $n\in\N$ elements up to equality, isomorphism, and isomorphism or
anti-isomorphism. Likewise, we give formulae for the number of
nilpotent commutative semigroups with $n$ elements up to equality and
up to isomorphism.
\end{abstract}
\maketitle


\section{Introduction}

The topic of enumerating finite algebraic or combinatorial objects of
a particular type is classical. Many theoretical enumeration results
were obtained thanks to the advanced orbit counting methods developed
by Redfield~\cite{Red27}, Polya~\cite{Pol37}, and de
Bruijn~\cite{Bru59}. Numerous applications of the method known as
power group enumeration can be found in~\cite{HP73}. Of particular
interest for this paper is the usage to count universal algebras
in~\cite{Har66}.

The enumeration of finite semigroups has mainly been performed by 
exhaustive search and the results are therefore restricted to very small
orders. The most recent numbers are of semigroups of order
9~\cite{Dis10}, of semigroups with identity of
order 10~\cite{DK09},
commutative semigroups of order 10~\cite{Gri03}, and
linearly ordered semigroups of order 7~\cite{Sla95}.

In this paper we use power group enumeration to develop formulae for
the number of semigroups of a particular type, which we now define.

A semigroup $S$ is \emph{nilpotent} if there exists a $r \in \N$ such
that the set 
\[
S^r=\set{s_1s_2\cdots s_r}{s_1, s_2, \ldots, s_r\in S}
\]
has size $1$. If $r$ is the least number such that $|S^r|=1$, then we
say that $S$ has \emph{(nilpotency) degree~$r$}.

As usual, the number of `structural types' of objects is of greater
interest than the number of distinct objects. Let $S$ and $T$ be
semigroups. Then a function $f:S\to T$ is
an \emph{isomorphism} if it is a bijection and $f(x\cdot y)=f(x)\cdot
f(y)$ for all $x,y\in S$. The \emph{dual} $S^*$ of $S$ is the
semigroup with multiplication $\ast$ defined by $x\ast y=y\cdot x$ on
the set $S$. A bijection $f:S
\to T$ is an \emph{anti-isomorphism} if $f$ is an isomorphism from $S^*$
to $T$. Throughout this article we distinguish between the number of
distinct semigroups on a set, the number up to isomorphism, and the
number up to isomorphism or anti-isomorphism. We shall refer to the
number of distinct semigroups that can be defined on a set as the
number \emph{up to equality}.

For $n\in\N$ we let $z(n)$ denote the number of nilpotent semigroups of
degree $3$ on $\{1,2,\ldots, n\}$. The particular interest in
nilpotent semigroups of degree $3$ stems from the observation that almost all
finite semigroups are of this type. More precisely, Kleitman,
Rothschild, and Spencer identified $z(n)$ in~\cite{KRS76} as an
asymptotic lower bound for the number of all semigroups on that
set. Furthermore, J\"urgensen, Migliorini, and Sz\'ep suspected
in~\cite{JMS91} that $z(n)/2n!$ was a good lower bound for the
number of semigroups with $n$ elements up to isomorphism or
anti-isomorphism based on the comparison of these two numbers for
$n=1,2, \ldots, 7$. This belief was later supported by Satoh, Yama,
and Tokizawa \cite[Section 8]{SYT94} and the first
author~\cite{Dis10} in their analyses of the semigroups with orders~8
and 9, respectively.

This paper is structured as follows: in the next
section we present and discuss our results, delaying certain technical details
for later sections; in Section~\ref{sec_constr} we describe a way to
construct semigroups of degree $2$ or $3$; in Section~\ref{sec_3nil}
nilpotent semigroups of degree $3$ are considered up to equality; in
Section~\ref{sec_PET} we present the relevant
background material from power group enumeration and a number of
technical results in preparation for Section~\ref{sec_proof} where we
give the proofs for our main theorems. Tables containing the first few
terms of the sequences defined by the various formulae in the paper
can be found at the appropriate points. The implementation used to obtain
these numbers is provided as the function {\tt
  Nr3NilpotentSemigroups} in the computer algebra system {\sf GAP}
\cite{GAP4} by the package {\sf Smallsemi}~\cite{smallsemi}. 

\section{Formulae for the number of nilpotent semigroups of degree $3$}

\subsection{Up to equality}
The number of nilpotent and commutative nilpotent semigroups of degree 3
on a finite set can be computed using formulae given
in~\cite[Theorems 15.3 and 15.8]{JMS91}. We summarise the relevant
results in the following theorem. As the theorems in~\cite{JMS91} are
stated incorrectly we shall give a proof for Theorem~\ref{lem_3nil_all} in
Section~\ref{sec_3nil}.
 
\begin{theorem}
\label{lem_3nil_all}
For $n\in\N$ the following hold:
\begin{enumerate} 
\item 
the number of distinct nilpotent semigroups of degree $3$ on
$\{1,2,\ldots, n\}$ is
\begin{equation*}
\sum_{m=2}^{a(n)}{n \choose m}m\sum_{i=0}^{m-1}(-1)^i{m-1 \choose
  i}(m-i)^{\left((n-m)^2\right)}
\end{equation*}
where $a(n)=\left\lfloor n+1/2-\sqrt{n-3/4}\,\right\rfloor$;
\item
the number of distinct commutative nilpotent semigroups of degree $3$ on
$\{1,2,\ldots, n\}$ is
\begin{equation*}
\sum_{m=2}^{c(n)}{n \choose m}m
\sum_{i=0}^{m-1}(-1)^i{m-1 \choose i}(m-i)^{(n-m)(n-m+1)/2}
\end{equation*}
where $c(n)=\left\lfloor n+3/2-\sqrt{2n+1/4}\,\right\rfloor.$
\end{enumerate}
\end{theorem}
Note that there are no nilpotent semigroups of degree $3$ with fewer than $3$ elements. 
 Accordingly, the formulae in Theorem \ref{lem_3nil_all} yield that  the number of nilpotent and commutative nilpotent semigroups of degree $3$ with $1$ or $2$ elements is $0$. The first few non-zero terms of the sequences given by Theorem \ref{lem_3nil_all} are shown in Tables ~\ref{tab_3nil} and \ref{tab_3nil_comm}.

\begin{table}
{\small
\caption{\label{tab_3nil} Numbers of nilpotent semigroups of degree $3$ up to equality}
\begin{tabular}{rl}
\toprule
$n$ & number of nilpotent semigroups of degree $3$ on $\{1,2,\ldots, n\}$ \\
\midrule
3&6\\
4&180\\
5&11\,720\\
6&3\,089\,250\\
7&5\,944\,080\,072\\
8&147\,348\,275\,209\,800\\
9&38\,430\,603\,831\,264\,883\,632\\
10&90\,116\,197\,775\,746\,464\,859\,791\,750\\
11&2\,118\,031\,078\,806\,486\,819\,496\,589\,635\,743\,440\\
12&966\,490\,887\,282\,837\,500\,134\,221\,233\,339\,527\,160\,717\,340\\
13&17\,165\,261\,053\,166\,610\,940\,029\,331\,024\,343\,115\,375\,665\,769\,316\,911\,576\\
14&6\,444\,206\,974\,822\,296\,283\,920\,298\,148\,689\,544\,172\,139\,277\,283\,018\,112\,679\,406\,098\,010\\
15&38\,707\,080\,168\,571\,500\,666\,424\,255\,328\,930\,879\,026\,861\,580\,617\,598\,218\,450\,546\,408\,004\,390\,044\,578\,120\\
\bottomrule
\end{tabular}
}
\end{table}

\begin{table}
{\small
\caption{\label{tab_3nil_comm} Numbers of commutative
  nilpotent semigroups of degree $3$ up to equality}
\begin{tabular}{rl}
\toprule
$n$ & number of commutative nilpotent semigroups of degree $3$ on $\{1,2,\ldots, n\}$\\
\midrule
3&6\\
4&84\\
5&1\,620\\
6&67\,170\\
7&7\,655\,424\\
8&2\,762\,847\,752\\
9&3\,177\,531\,099\,864\\
10&11\,942\,816\,968\,513\,350\\
11&170\,387\,990\,514\,807\,763\,280\\
12&11\,445\,734\,473\,992\,302\,207\,677\,404\\
13&3\,783\,741\,947\,416\,133\,941\,828\,688\,621\,484\\
14&5\,515\,869\,594\,360\,617\,154\,295\,309\,604\,962\,217\,274\\
15&33\,920\,023\,793\,863\,706\,955\,629\,537\,246\,610\,157\,737\,736\,800\\
16&961\,315\,883\,918\,211\,839\,933\,605\,601\,923\,922\,425\,713\,635\,603\,848\,080\\
17&160\,898\,868\,329\,022\,121\,111\,520\,489\,011\,089\,643\,697\,943\,356\,922\,368\,997\,915\,120\\
\bottomrule
\end{tabular}
}
\end{table}

\subsection{Up to isomorphism and up to isomorphism or anti-isomorphism}
Our main results are explicit formulae for the number of nilpotent and
commutative nilpotent semigroups of degree $3$ on any finite set up to
isomorphism and up to isomorphism or anti-isomorphism. As every
commutative semigroup is equal to its dual we obtain three different
formulae. 

If $j$ is a partition of $n\in \N$, written as $j \vdash n$, then we denote
by $j_i$ the number of summands equalling $i$. The first of our
main theorems, dealing with nilpotent semigroups of degree $3$ up to
isomorphism, can then be stated as follows:

\begin{theorem}
\label{thm_up_to_iso}
Let $n,p,q\in\N$. For $1\leq q < p$ denote
\begin{equation}
\label{eq_N}
N(p,q) = \sum_{j \vdash q-1} \sum_{k \vdash p-q} \left(\prod_{i=1}^{q-1}
  j_i!\,i^{j_i} \prod_{i=1}^{p-q} k_i!\,i^{k_i}\right)^{-1} \prod_{a,b=1}^{p-q}
\left(1+ \sum_{d \mid \lcm(a,b)} dj_d
\right)^{k_{a}k_{b}\gcd(a,b)}.
\end{equation}
Then the number of nilpotent semigroups of degree $3$ and order $n$ up to
isomorphism equals
\begin{equation*}
\sum_{m=2}^{a(n)}\left(N(n,n)-N(n-1, m-1)\right)\mbox{ \ where \ } 
a(n)=\left\lfloor n+1/2-\sqrt{n-3/4}\,\right\rfloor,
\end{equation*}
\end{theorem}

\begin{table}
{\small
\caption{\label{tab_3nil_uptoiso} Numbers of nilpotent semigroups of
  degree $3$ up to isomorphism}
\begin{tabular}{rl}
\toprule
$n$ & number of non-isomorphic nilpotent semigroups of degree $3$ of order $n$\\
\midrule
 3 & 1\\
 4 & 9\\
 5 & 118\\
 6 & 4\,671\\
 7 & 1\,199\,989\\
 8 & 3\,661\,522\,792\\
 9 & 105\,931\,872\,028\,455\\
10 & 24\,834\,563\,582\,168\,716\,305\\
11 & 53\,061\,406\,576\,514\,239\,124\,327\,751\\
12 & 2\,017\,720\,196\,187\,069\,550\,262\,596\,208\,732\,035\\
13 & 2\,756\,576\,827\,989\,210\,680\,367\,439\,732\,667\,802\,738\,773\,384\\
14 & 73\,919\,858\,836\,708\,511\,517\,426\,763\,179\,873\,538\,289\,329\,852\,786\,253\,510\\
15 & 29\,599\,937\,964\,452\,484\,359\,589\,007\,277\,447\,538\,854\,227\,891\,149\,791\,717\,673\,581\,110\,642\\
\bottomrule
\end{tabular}
}
\end{table}

The second of our main theorems gives the number of nilpotent semigroups of
degree $3$ up to isomorphism or anti-isomorphism.

\begin{theorem}
\label{thm_up_to_equi}
Let $n,p,q\in\N$. For $1\leq q < p$ let $N(p,q)$ as in \eqref{eq_N}
and denote
\begin{equation}
\label{eq_L}
L(p,q) = \frac{1}{2}N(p,q) + \frac{1}{2}
\sum_{j \vdash q-1} \sum_{k \vdash p-q}
\left(\prod_{i=1}^{q-1}
  j_i!\,i^{j_i} \prod_{i=1}^{p-q} k_i!\,i^{k_i}\right)^{-1}
\prod_{a=1}^{p-q}\left(
q_a^{k_a}p_{a,a}^{k_a^2-k_a}\prod_{b=1}^{a-1} p_{a,b}^{2k_ak_b}
\right),
\end{equation}
where 
\[
p_{a,b}=\left(1+ \sum_{d \mid \lcm(2,a,b)} dj_d \right)^{ab/\lcm(2,a,b)}
\]
and
\[
q_a=
\begin{cases}
(1+ \sum_{d \mid a} dj_d)(1+ \sum_{d \mid 2a} dj_d)^{(a-1)/2}
& \text{ if \ } a \equiv 1 \mod 2\\
(1+ \sum_{d \mid a} dj_d)^a 
& \text{ if \ }  a \equiv 0 \mod 4\\
(1+ \sum_{d \mid a/2} dj_d)^2(1+ \sum_{d \mid a} dj_d)^{a-1} 
& \text{ if \ }  a \equiv 2 \mod 4.
\end{cases}
\]
Then the number of nilpotent semigroups of degree $3$ and order $n$ up to
isomorphism or anti-isomorphism equals
\begin{equation*}
\sum_{m=2}^{a(n)}\left(L(n,m)-L(n-1,m-1)\right)\mbox{ \ where \ }
a(n)=\left\lfloor n+1/2-\sqrt{n-3/4}\,\right\rfloor.
\end{equation*}
\end{theorem}

\begin{table}
{\small
\caption{\label{tab_3nil_uptoequiv}Numbers of nilpotent semigroups of
  degree $3$ up to isomorphism or anti-isomorphism}
\begin{tabular}{rl}
\toprule
 $n$& number of non-(anti)-isomorphic nilpotent semigroups of degree $3$ of order $n$\\
\midrule
 3 &  1 \\
 4 &  8 \\
 5 &  84 \\
 6 &  2\,660 \\
 7 &  609\,797 \\
 8 &  1\,831\,687\,022 \\
 9 &  52\,966\,239\,062\,973 \\
 10 &  12\,417\,282\,095\,522\,918\,811 \\
 11 &  26\,530\,703\,289\,252\,298\,687\,053\,072 \\
 12 &  1\,008\,860\,098\,093\,547\,692\,911\,901\,804\,990\,610 \\
 13 &  1\,378\,288\,413\,994\,605\,341\,053\,354\,105\,969\,660\,808\,031\,163
\\
 14 &  36\,959\,929\,418\,354\,255\,758\,713\,676\,933\,402\,538\,920\,157\,765\,946\,956\,889\\
 15 &  14\,799\,968\,982\,226\,242\,179\,794\,503\,639\,146\,983\,952\,853\,044\,950\,740\,907\,666\,303\,436\,922\\
\bottomrule
\end{tabular}
}
\end{table}

A semigroup is \emph{self-dual} if it is isomorphic to its
dual. The concept of anti-isomorphism has no relevance for self-dual
semigroups. Combining Theorems~\ref{thm_up_to_iso} and~\ref{thm_up_to_equi},
it is possible to deduce a formula for the number of self-dual,
nilpotent semigroups of degree $3$ up to isomorphism. More generally,
considering semigroups of a certain type the number of self-dual
semigroups up to isomorphism is equal to twice the number of
semigroups up to isomorphism and anti-isomorphism minus the number of
semigroups up to isomorphism. 

\begin{corollary}
\label{coro_selfdual}
Let $n\in\N$ and let $N(p,q)$ and $L(p,q)$ be as defined in \eqref{eq_N} and \eqref{eq_L}, respectively.
Then the number of self-dual, nilpotent semigroups of degree $3$ and
order $n$ up to isomorphism equals
\begin{multline*}
\sum_{m=2}^{a(n)}\left(2L(n,m)-N(n,m)-2L(n-1,m-1)+N(n-1,m-1)\right)
\\\mbox{ \ where \ }a(n)=\left\lfloor n+1/2-\sqrt{n-3/4}\,\right\rfloor.
\end{multline*}
\end{corollary}

\begin{table}
{\small
\caption{\label{tab_3nil_selfdual}Numbers of self-dual nilpotent
  semigroups of degree $3$ up to isomorphism}
\begin{tabular}{rl}
\toprule
$n$ & number of non-isomorphic self-dual nilpotent semigroups of degree $3$ of
order $n$\\
\midrule
 3 & 1\\
 4 & 7\\
 5 & 50\\
 6 & 649\\
 7 & 19\,605\\
 8 & 1\,851\,252\\
 9 & 606\,097\,491\\
10 & 608\,877\,121\,317\\
11 & 1\,990\,358\,249\,778\,393\\
12 & 25\,835\,561\,207\,401\,249\,185\\
13 & 1\,739\,268\,479\,271\,518\,877\,288\,942\\
14 & 590\,686\,931\,539\,550\,985\,679\,107\,660\,268\\
15 & 846\,429\,051\,478\,198\,751\,690\,097\,659\,025\,763\,202\\
\bottomrule
\end{tabular}
}
\end{table}

Substituting in the previous corollary the actual formula for
$2L(p,q)$ we notice that $N(p,q)/2$ appears as term in $L(p,q)$ and
cancels. The resulting simplified formula is implemented as part of
the function {\tt Nr3NilpotentSemigroups} in {\sf Smallsemi}
\cite{smallsemi}.
 
Since commutative semigroups are self-dual, we obtain just one formula up
to isomorphism for commutative nilpotent semigroups of degree $3$.

\begin{theorem}
\label{thm_comm}
Let $n,p,q\in\N$. For $1\leq q < p$ denote
\begin{multline*}
K(p,q)=
\sum_{j \vdash q-1} \sum_{k \vdash p-q} \left[
\left(\prod_{i=1}^{q-1} j_i!\,i^{j_i}\prod_{i=1}^{p-q} k_i!\,i^{k_i}\right)^{-1}
\prod_{a=1}^{\lfloor \frac{n}{2} \rfloor} 
\left(1+ \sum_{d \mid a} dj_d\right)^{k_{2a}}
\left(1+ \sum_{d \mid 2a} dj_d\right)^{a{k_{2a}}}\cdot\right. \\
\left. \prod_{a=1}^{\lfloor \frac{n+1}{2} \rfloor}
\left(1+ \sum_{d \mid 2a-1} dj_d\right)^{ak_{2a-1}}
\prod_{a<b}
\left(1+ \sum_{d \mid \lcm(a,b)} dj_d\right)^{k_{a}k_{b}\gcd(a,b)}\right].
\end{multline*}
Then the number of nilpotent, commutative semigroups of degree $3$ and order
$n$ up to isomorphism equals
\begin{equation*}
\sum_{m=2}^{c(n)}\left(K(n,m)-K(n-1,m-1)\right)\mbox{ \ where \ }
c(n)=\left\lfloor n+3/2-\sqrt{2n+1/4}\,\right\rfloor.
\end{equation*}
\end{theorem}

\begin{table}
{\small
\caption{\label{tab_3nil_comm_upto}Numbers of commutative nilpotent
  semigroups of degree $3$ up to isomorphism}
\begin{tabular}{rl}
\toprule
$n$ & number of non-isomorphic commutative nilpotent semigroups of degree $3$
of order $n$\\
\midrule
 3 & 1\\
 4 & 5\\
 5 & 23\\
 6 & 155\\
 7 & 2\,106\\
 8 & 79\,997\\
 9 & 9\,350\,240\\
10 & 3\,377\,274\,621\\
11 & 4\,305\,807\,399\,354\\
12 & 23\,951\,673\,822\,318\,901\\
13 & 608\,006\,617\,857\,847\,433\,462\\
14 & 63\,282\,042\,551\,031\,180\,915\,403\,659\\
15 & 25\,940\,470\,166\,038\,603\,666\,194\,391\,357\,972\\
16 & 45\,946\,454\,978\,824\,286\,601\,551\,283\,052\,739\,171\,318\\
17 & 452\,361\,442\,895\,926\,947\,438\,998\,019\,240\,982\,893\,517\,749\,169\\
18 & 30\,258\,046\,596\,218\,438\,115\,657\,059\,107\,812\,634\,405\,962\,381\,166\,457\,711\\
19 & 12\,094\,270\,656\,160\,403\,920\,767\,935\,604\,624\,748\,908\,993\,169\,949\,317\,454\,767\,617\,795\\
\bottomrule
\end{tabular}
}
\end{table}

To determine the number of nilpotent semigroups of degree $3$ up to
isomorphism or up to isomorphism or anti-isomorphism, we use the
technique of power group enumeration in a similar way as Harrison did
for universal algebras~\cite{Har66}. In Section~\ref{sec_PET} we
present the relevant background material and a number of technical
results in preparation for Section~\ref{sec_proof} where we give the
proofs for Theorems~\ref{thm_up_to_iso}, \ref{thm_up_to_equi}, and
\ref{thm_comm}.

\subsection{Bounds and asymptotics}

\begin{table}
{\small
\caption{\label{tab_compare}Numbers of semigroups and nilpotent
  semigroups of degree $3$}
\begin{tabular}{rrrr}
\toprule
& number of semigroups & number of semigroups of\\
$n$ & up to isomorphism & degree $3$ up to isomorphism & lower bound $\lceil
z(n)/2n! \rceil$\\
& or anti-isomorphism & or anti-isomorphism\\
\midrule
3 & 18 & 1 & 1\\
4 & 126 & 8 & 4\\
5 & 1\,160 & 84 & 49\\
6 & 15\,973 & 2\,660 & 2\,146\\
7 & 836\,021 & 609\,797 & 589\,691\\
8 & 1\,843\,120\,128 & 1\,831\,687\,022 & 1\,827\,235\,556 \\
9 & 52\,989\,400\,714\,478 & 52\,966\,239\,062\,973 &
52\,952\,220\,887\,436\\
10 & {\it unknown} & 12\,417\,282\,095\,522\,918\,811 &
12\,416\,804\,146\,790\,463\,082\\
\bottomrule
\end{tabular}
}
\end{table}

The formula for the number of nilpotent semigroups of degree $3$ up to
isomorphism or anti-isomorphism in Theorem~\ref{thm_up_to_equi}
provides a new lower bound for the number of semigroups up to
isomorphism or anti-isomorphism of a given size. Presumably this bound
is asymptotic, that is, the ratio tends to $1$ while the order tends to
infinity, although this is not a consequence of the result for
semigroups up to equality in~\cite{KRS76}. The comparison in
Table~\ref{tab_compare} shows also that the lower bound $z(n)/2n!$
from \cite{JMS91} seems to converge rapidly towards our new bound. 
Analogous observations can be made considering only commutative
semigroups though the convergence appears slower as mentioned by
Grillet in the analysis in~\cite{Gri03}.

Our formulae also yield a large qualitative improvement over the old
lower bound since they give exact numbers of nilpotent semigroups of
degree $3$. In particular,  the provided numbers can be used to cut down
the effort required in an exhaustive search to determine the number of
semigroups of a given order, as already done for semigroups of
order 9 in~\cite{Dis10}.

The conjectured asymptotic behaviour of the lower bound of $z(n)/2n!$
for the number of semigroups of order $n$  would imply
that almost all sets of isomorphic semigroups on $\{1,2,\dots,n\}$ are
of size $n!$. In other words, most semigroups have trivial
automorphism group; a property that is known for various types of
algebraic and combinatorial objects, for example graphs
\cite{ER63}. Our formulae could help to prove this conjecture at least
for nilpotent semigroups of degree $3$. In each summand in \eqref{eq_N}
those semigroups of degree $3$ are counted for which a bijection with
cycle structure corresponding to the partitions $j$ and $k$ is an
automorphism. It remains to estimate the contribution of all summands
that do not correspond to the identity map.


\section{Construction of nilpotent semigroups of degree 2 or 3}
\label{sec_constr}

In this section we describe how to construct nilpotent semigroups of
degree 2 or 3 on an $n$-element set. A similar construction is given in
\cite{KRS76}. For the sake of brevity we will denote by $[n]$
the set $\{1,2,\ldots, n\}$ where $n\in \N$.

\begin{defn}\label{defn_nil}
Let $n \geq 2$, let $A$ be a non-empty proper subset of $[n]$, and let $B$ denote the complement of 
$A$ in $[n]$. If $z\in B$ is arbitrary and $\psi: A \times A \to B$ is any function, then we can define 
multiplication on $[n]$ by
\begin{equation}
\label{eq_constr}
x y= 
\begin{cases}
\psi(x,y)&\text{if }x,y\in A\\
z&\text{otherwise.}
\end{cases}
\end{equation}
We will denote the set $[n]$ with the operation given above by $\U(A, \psi, z)$. 
\end{defn}

Any product $a b c$ in $\U(A, \psi,z)$ equals
$z$, and so the multiplication defined in \eqref{eq_constr} is associative. It follows that 
$\U(A, \psi, z)$ is a nilpotent  semigroup of degree 2 or 3. The
semigroup $\U(A, \psi, z)$ has degree 2 if and only if $\U(A, \psi, z)$ is a zero
semigroup if and only if $\psi$ is the constant function with value
$z$. Conversely, if $T$ is a nilpotent semigroup of degree $3$ with elements $[n]$, then setting
$A=T\setminus T^2$, letting $\psi:A\times A\to T^2$ be defined by $\psi(x,y)=xy$
for all $x,y\in T$, and setting $z$ to be the zero element of $T$, we see that
$T=\U(A, \psi, z)$.
Therefore when enumerating nilpotent semigroups of degree $3$ it
suffices to consider the semigroups $\U(A, \psi,z)$.


\section{Semigroups and commutative semigroups of degree $3$ up to
  equality}
\label{sec_3nil}

Denote by $Z_n$ the set of nilpotent semigroups of degree 3 on
$\{1,2,\ldots, n\}$. A formula for the cardinality of a
proper subset of $Z_n$ is stated in Theorem 15.3 of \cite{JMS91}. 
However, the formula given in \cite{JMS91} actually yields $|Z_n|$ and
this is what the proof of the theorem in \cite{JMS91} shows. Similarly,
the formula in Theorem 15.8 of \cite{JMS91} can be used to determine
the number of all commutative semigroups in $Z_n$ even though the
statement says otherwise. For the sake of completeness and to avoid
confusion we prove that the formulae as given in
Theorem~\ref{lem_3nil_all} are correct.

\noindent\proofref{lem_3nil_all}
In both parts of the proof,  we let $A$ be a fixed non-empty proper subset of $[n]=\{1,2,\ldots, n\}$, let $B$ denote the complement of $A$ in $[n]$, let $m=|B|$, and let $z\in B$ be fixed. We consider  semigroups of the form $\U(A, \psi, z)$ where $\psi:A\times A\to B$ as given in Definition \ref{defn_nil}.
 
{\bf (i).} The number of functions from $A\times A$ to
$B$ is $m^{(n-m)^2}$. To avoid counting semigroups twice for different $m$,
we will only consider those functions $\psi$ where every element in $B\setminus\{z\}$ appears in the image of $\psi$. For a subset $X$ of $B\setminus\{z\}$ of size $i$,
there are $(m-i)^{(n-m)^2}$ functions with no element from $X$ in their image. Using the  Inclusion-Exclusion Principle, the number of functions from $A\times A$ to $B$ with image containing   $B\setminus\{z\}$ is 
\begin{equation}
\label{count_different}
\sum_{i=0}^{m-1}(-1)^i{m-1 \choose i}(m-i)^{(n-m)^2}.
\end{equation}

The function $\psi$ is defined on a set with $(n-m)^2$ elements. Hence
the condition that $B\setminus\{z\}$ is contained in the image of $\psi$ 
implies that $m-1 \leq (n-m)^2$. Reformulation yields
\begin{equation}\label{eq_unequal}
m \leq n+1/2-\sqrt{n-3/4}.
\end{equation}
If $m=1$, then every function $\psi:A\times A\to B$ is constant, and
so, as mentioned above, $\U(A, \psi, z)$ is not nilpotent of degree $3$.
Summing \eqref{count_different} over all  appropriate values of $m$, the ${n\choose m}$ choices for $B$ and the $m$ choices for $z\in B$ concludes the proof of this part.

{\bf (ii).}    
If  $\U(A, \psi,z)$ is a commutative semigroup, then the function
$\psi:A\times A \to B$ is defined by its values on pairs
$(i,j)$ with $i \leq j$. There are
$(n-m)(n-m+1)/2$ such pairs and hence there are $m^{(n-m)(n-m+1)/2}$ such functions $\psi$. 

The rest of the proof follows the same steps as the proof of part (i)
with $m^{(n-m)(n-m+1)/2}$ replacing $m^{(n-m)^2}$ and where  the
inequality $m-1 \leq (n-m)(n-m+1)/2$ yields the parameter $c(n)$.
\qed


\section{Power group enumeration}
\label{sec_PET}
In this section, we shall introduce the required background material
relating to power group enumeration and determine the cycle indices of
certain power groups necessary to prove our main theorems. The
presentation in this section is based on \cite{HP73}.
 
Let $X$ be a non-empty set and let $S_X$ denote the symmetric group on
$X$. We again denote the set $\{1,2,\ldots, n\}$ by $[n]$, and will
write $S_n$ instead of $S_X$ if $X=[n]$. For a permutation $\pi\in
S_X$, let $\delta(\pi,k)$ denote the number of cycles of length $k$ in
the disjoint cycle decomposition of $\pi$. 

\begin{defn}\label{def_cyc_ind}
Let $G$ be a subgroup of {\color{blue} $S_n$}. Then the polynomial
\[
\mathcal{Z}(G;x_1,x_2,\ldots,x_n) = 
\frac{1}{|G|} \sum_{g \in G}\prod_{k=1}^n x_k^{\delta(g,k)}
\]
is called  the \emph{cycle index} of the group $G$; in short, we write $\mathcal{Z}(G)$.
\end{defn}

The cycle structure of a permutation $\pi\in S_n$ corresponds to a
partition of $n$, and all elements with the same cycle structure form
a conjugacy class of $S_n$. Remember that if $j$ is a partition of
$n$, written as $j \vdash n$, then we denote by $j_i$ the number of
summands equalling $i$. This yields $j_i=\delta(\pi,i)$ for all $i$
and for each element $\pi$ in the conjugacy class corresponding to
$j$. This observation allows us to write the cycle index of the
symmetric group in a compact form.

\begin{lemma}[{\cite[(2.2.5)]{HP73}}]
\label{lem_cycind_Sn}
The cycle index of $S_n$ is 
\[
\mathcal{Z}(S_n)=\sum_{j\vdash
  n}\left(\prod_{i=1}^nj_i!i^{j_i}\right)^{-1} \prod_{a=1}^nx_a^{j_a}.
\]
\end{lemma}

In what follows we require actions other than the natural action of the
symmetric group $S_X$ on $X$. In particular, we require actions  on
functions in which two groups act independently on the domains
and on the images of the functions. If $G$ is a group acting on a set
$X$, then we denote by $x^g$ the image of $x\in X$ under the action of $g\in G$. 

\begin{defn}
\label{def_power_grp}
Let $A$ and $B$ be subgroups of $S_X$ and $S_Y$, respectively, where 
 $X$ and $Y$ are finite disjoint sets.
Then we define an action of the group $A\times B$ on the set $Y^X$ of functions from $X$ to $Y$ in the following way: the image of $f \in Y^X$ under $(\alpha,\beta)\in A\times B$ is given by
\[
f^{(\alpha,\beta)}(x)=\left(f\left(x^\alpha\right)\right)^\beta
\]
for all $x\in X$. We will refer to $A\times B$ with this action as a \emph{power group}.
\end{defn}

The cycle index itself is not required for the power groups used in this paper. 
Of interest is the constant form of the Power Group
Enumeration Theorem given below, which states the number of
orbits under the action of a power group. The result
goes back to de Bruijn~\cite{Bru59}, but is presented here in the form
given in~\cite[Section 6.1]{HP73}.

\begin{theorem}\label{thm_PGET}
Let $A\times B$ be a power group acting on the functions $Y^X$ as in Definition 
\ref{def_power_grp}.
Then the number of orbits  of $A\times B$ on $Y^X$  
equals
$$\frac{1}{|B|}\sum_{\beta\in B}\mathcal{Z}(A;c_1(\beta),c_2(\beta),\dots,c_{|X|}(\beta)),$$
where
$$c_i(\beta)=\sum_{d\mid i}d\,\delta(\beta,d).$$
\end{theorem}

To apply Theorem~\ref{thm_PGET} in the enumeration of nilpotent
semigroups of degree $3$ we require the cycle indices of the specific
group actions defined below. 

\begin{defn}\label{defn_actions}
Let $A$ be a group acting on a set $X$. Then we define:
\begin{enumerate}
\item by $A^{\times 2}$ 
the group $A$ acting on $X\times X$ 
componentwise, that is, 
$$(x_1,x_2)^{\alpha} =(x_1^{\alpha},x_2^{\alpha})$$ for $\alpha\in A$;

\item by $2A^{\times 2}$ the group $S_2\times A$  acting on $X\times X$ by 
$$(x_1, x_2)^{(\pi, \alpha)}=(x_{1^\pi}^\alpha, x_{2^\pi}^\alpha)$$
for $\alpha\in A$ and $\pi\in S_2$. 
\item
by $A^{\{2\}}$ the group $A$ acting pointwise on the set 
$\left\{ \{x_1,x_2\} \mid x_i \in X\right\}$ of subsets of a set $X$ with $1$ or $2$ elements, that is, 
$$\{x_1,x_2\}^{\alpha} =\{x_1^{\alpha},x_2^{\alpha}\}$$ for $\alpha\in A.$
\end{enumerate}
\end{defn}


We will show in Section \ref{sec_proof} that it is possible to
distinguish nilpotent semigroups of degree $3$ of the form
$\U(A,\psi,z)$ as defined in Definition \ref{defn_nil} up to isomorphism, and up to 
isomorphism or anti-isomorphism, by determining the orbit the function
$\psi$ belongs to under certain power 
groups derived from the actions in Definition \ref{defn_actions}.  

In the next lemma, we obtain the cycle indices of the groups
$S_n^{\times 2}, S_n^{\{2\}},$ and $2S_n^{\times 2}$ using the cycle
index of $S_n$ given in Lemma \ref{lem_cycind_Sn}. 

\begin{lemma}
\label{lem_cyc_ind}
For $n\in\N$ the following hold:
\begin{enumerate}
\item
the cycle index of $S_n^{\times 2}$ is
\begin{equation*}
\mathcal{Z}(S^{\times 2}_n) = \sum_{j \vdash n}
\left(\prod_{i=1}^n j_i!\,i^{j_i}\right)^{-1}
\prod_{a,b=1}^n x_{\lcm (a,b)}^{j_aj_b\gcd(a,b)};
\end{equation*}
\item
the cycle index of $2S_n^{\times 2}$ is
\begin{equation*}
\label{Z_twisted}
\mathcal{Z}(2S^{\times 2}_n) = \frac{1}{2}
\mathcal{Z}(S_n^{\times 2}) + \frac{1}{2}
\sum_{j \vdash n} \left(\prod_{i=1}^n j_i!\,i^{j_i}\right)^{-1}
\prod_{a=1}^{n}
\left(
q_a^{j_a}p_{a,a}^{j_a^2-j_a}\prod_{b=1}^{a-1} p_{a,b}^{2j_aj_b}
\right), 
\end{equation*}
where $p_{a,b}=x_{\lcm(2,a,b)}^{ab/\lcm(2,a,b)}$ and
\[
q_a=
\begin{cases}
x_ax_{2a}^{(a-1)/2} & \text{ if \ } a \equiv 1 \mod 2\\
x_a^a & \text{ if \ }  a \equiv 0 \mod 4\\
x_{a/2}^2x_a^{a-1} & \text{ if \ }  a \equiv 2 \mod 4;
\end{cases}
\]
\item
the cycle index of $S_n^{\{2\}}$ is
\begin{equation*}
\label{Z_S{2}}
\mathcal{Z}(S^{\{2\}}_n)=\sum_{j\vdash  n} \left( \prod_{i=1}^{n}j_i!i^{j_i}
\right)^{-1} \prod_{a=1}^{\lfloor n/2 \rfloor}r_a
\prod_{a=1}^{\lfloor (n+1)/2 \rfloor}s_a 
\prod_{a=1}^n t_a
\left(\prod_{b=1}^{a-1}x_{\lcm(a,b)}^{j_aj_b\gcd(a,b)}\right),
\end{equation*}
where the monomials are $r_a=x_a^{j_{2a}}x_{2a}^{aj_{2a}}$,
$s_a=x_{2a-1}^{aj_{2a-1}}$, and $t_a=x_a^{a(j_a^2-j_a)/2}$.
\end{enumerate}
\end{lemma}
\begin{proof}
{\bf (i).}
By definition each permutation in $S_n$ induces a permutation in
$S_n^{\times 2}$. 
Let $\alpha\in S_n$ and let $z_a$ and $z_b$ be two cycles
thereof with length $a$ and $b$ respectively. Consider the action of
$\alpha$ on those pairs in $[n]\times [n]$ which have as first component an
element in $z_a$ and as second component an element in $z_b$. Let
$(i,j)\in [n]\times[n]$ be one such pair. Since $i^{\alpha^k}=i$ if and only
if $a$ divides $k$, and  $j^{\alpha^k}=j$ if and only if $b$ divides
$k$, the pair $(i,j)$ is in an orbit of length
$\lcm(a,b)$. The total number of pairs with first component in $z_a$
and second component in $z_b$ equals $ab$. Hence the number of orbits equals
$\gcd(a,b)$. Repeating this consideration for every pair of
cycles in $\alpha$ leads to 
$$\prod_{a,b=1}^n
x_{\lcm(a,b)}^{\delta(\alpha,a) \delta(\alpha,b) \gcd(a,b)}$$ 
as the summand corresponding to $\alpha$ in the cycle index
$\mathcal{Z}(S_n^{\times 2})$. This yields
\[
\mathcal{Z}(S_n^{\times 2})= \frac{1}{n!}\sum_{\alpha \in S_n} \prod_{a,b=1}^n
x_{\lcm(a,b)}^{\delta(\alpha,a) \delta(\alpha,b) \gcd(a,b)}.
\]
That the contribution of $\alpha$ to $\mathcal{Z}(S_n^{\times 2})$ only depends on its cycle
structure allows us to replace the summation over all group
elements by a summation over partitions of $n$; one for each conjugacy
class of $S_n$. The number of elements with cycle structure associated to a
partition $j\vdash n$ equals $n!/\prod_{i=1}^n j_i!\,i^{j_i}$. Therefore 
\[
\mathcal{Z}(S_n^{\times 2})= \frac{1}{n!}\sum_{j\vdash n}
\frac{n!}{\prod_{i=1}^n j_i!\,i^{j_i}}
\prod_{a,b=1}^n x_{\lcm(a,b)}^{j_a j_b \gcd(a,b)},
\]
and cancelling the factor $n!$ concludes the proof.

{\bf (ii).}
For elements  $(\id_{\{1,2\}}, \alpha)\in 2S^{\times 2}_n$  the contribution to the cycle
index of $2S^{\times 2}_n$ equals the contribute of $\alpha$ to $\mathcal{Z}(S^{\times 2}_n)$ given in (i). It is rearranged as follows to
illustrate which contributions come from identical cycles and which
from disjoint cycles:
\[
\prod_{a,b=1}^n x_{\lcm(a,b)}^{\delta(\alpha,a) \delta(\alpha,b) \gcd(a,b)}
= \prod_{a=1}^{n}
\left(
x_a^{a\delta(\alpha,a)}x_a^{a(\delta(\alpha,a)^2-\delta(\alpha,a))}\prod_{b<a} 
x_{\lcm(a,b)}^{2\delta(\alpha,a)\delta(\alpha,b)\gcd(a,b)}
\right).
\]

For group elements of the form $((1\,2), \alpha)$ the contribution is
going to be deduced from the one of $\alpha$. Let $z_a$ and $z_b$
again be two cycles in $\alpha$ of length $a$ and
$b$ respectively, and assume at first, they are disjoint. Then $z_a$
and $z_b$ induce $2\gcd(a,b)$ orbits of length $\lcm(a,b)$ on the $2ab$
pairs in $[n]\times[n]$ with one component from each of the two cycles. Let 
\begin{equation}
\label{eq_omega}
\omega = \left\{(i_1,j_1),(i_2,j_2),\dots,
(i_{\lcm(a,b)},j_{\lcm(a,b)})\right\}
\end{equation}
be such an orbit. Then
\begin{equation}
\label{eq_baromega}
\bar{\omega} = \left\{(j_1,i_1),(j_2,i_2),\dots,
(j_{\lcm(a,b)},i_{\lcm(a,b)})\right\}
\end{equation}
is another one. The set $\omega
\cup \bar{\omega}$ is closed under the action of $((1\,2), \alpha)$.
In how many orbits $\omega \cup \bar{\omega}$ splits depends on the
parity of $a$ and $b$. Acting with $((1\,2), \alpha)$ on $(i_1,j_1)$ for
$\lcm(a,b)$ times gives $(i_1,j_1)$ if $\lcm(a,b)$ is even and
$(j_1,i_1)$ if $\lcm(a,b)$ is odd. Hence the two orbits $\omega$ and
$\bar{\omega}$ merge to one orbit in the latter case and give two
new orbits of the original length otherwise. This yields the monomial
\begin{equation*}
x_{\lcm(2,a,b)}^{2ab/\lcm(2,a,b)} =
\begin{cases}
x_{\lcm(a,b)}^{2\gcd(a,b)} \mbox{ \ if \ } \lcm(a,b) \equiv 0 \bmod 2\\
x_{2\lcm(a,b)}^{\gcd(a,b)} \mbox{ \ if \ } \lcm(a,b) \equiv 1 \bmod 2,
\end{cases}
\end{equation*}
which appears $\delta(\alpha,a)\delta(\alpha,b)$ times if $a\neq b$
and $(\delta(\alpha,a)^2-\delta(\alpha,a))/2$ times if $a=b$.

Let $z_a$ and $z_b$ now be identical and equal to the cycle $(i_1 i_2
\cdots i_a)$. The contribution to the monomial of $\alpha$ is the
factor $x_a^a$. The orbits are of the form $\{(i_g,i_h) \mid 1 \leq
g,h \leq a, g \equiv h+s \bmod a\}$ for $0 \leq s \leq a-1$. For $s=0$ the
orbit consists of pairs with equal entries, that is, $\{(i_1,i_1),
(i_2,i_2)\dots (i_a,i_a)\}$, and thus stays the same under $((1\,2), \alpha)$.
For an orbit $\omega =\{(i_g,i_h) \mid 1 \leq g,h \leq a, g \equiv h+s
\bmod a\}$ with $s \neq 0$ define $\bar{\omega}$ as in (\ref{eq_baromega}).
 If $\omega \neq
\bar{\omega}$ one argues like in the case of two disjoint cycles and
gets the result depending on the parity of $a$. Note that $\omega =
\bar{\omega}$ if and only if $s=a/2$. In particular this does
not occur for $a$ odd in which case
$$
x_ax_{2a}^{(a-1)/2}
$$
is the factor contributed to the monomial of $((1\,2), \alpha)$. If on the
other hand $a$ is even, one more case split is needed to deal with the
orbit 
$$
\omega = \{(i_g,i_h) \mid 1 \leq g,h \leq a, g \equiv h+a/2 \bmod a\}.
$$
Acting with $((1\,2), \alpha)$ on $(i_a,i_{a/2})$ for $a/2$ times gives
$(i_a,i_{a/2})$ if $a/2$ is odd and $(i_{a/2},i_a)$ if $a/2$ is
even. Thus $\omega$ splits into two orbits of length $a/2$ in the
former case and stays one orbit in the latter. The resulting factors
contributed to the monomial of $((1\,2), \alpha)$ are therefore
\begin{eqnarray*}
x_a^a & \mbox{ if } &  a \equiv 0 \mod 4\\
x_{a/2}^2x_a^{a-1} & \mbox{ if } &  a \equiv 2 \mod 4.
\end{eqnarray*}
Following the analysis for all pairs of cycles in $\alpha$ leads to
the contribution of $((1\,2), \alpha)$ to the cycle index. Summing as
before over all partitions of $n$, which correspond to the
different cycle structures, proves the formula for
$\mathcal{Z}(2S^{\times 2}_n)$.

{\bf (iii).}
To compute $\mathcal{Z}(S^{\{2\}}_n)$ let $\omega$ and $\bar{\omega}$
as in (\ref{eq_omega}) and (\ref{eq_baromega}) be orbits for two
cycles $z_a$ and $z_b$ from $\alpha\in S_n$ acting on $[n]\times [n]$.
If the two cycles $z_a$ and $z_b$ are disjoint then both $\omega$ and
$\bar{\omega}$ correspond to the same orbit
\[
\left\{\{i_1,j_1\},\{i_2,j_2\},\dots,\{i_{\lcm(a,b)},j_{\lcm(a,b)}\}\right\}
\]
of $\alpha$ acting on $[n]^{\{2\}}$. The contribution to the monomial of
$\alpha$ in $\mathcal{Z}(S^{\{2\}}_n)$ is therefore $x_{\lcm(a,b)}^{\gcd(a,b)}$. Let $z_a$ and $z_b$ now be
identical and equal to the cycle $(i_1 i_2\cdots i_a)$. In
$S^{\times 2}_n$ this gave rise to the orbits $\{(i_g,i_h) \mid 1
\leq g,h \leq a, g \equiv h+s \bmod a\}$ for $0 \leq s \leq a-1$. The
corresponding orbit under $S^{\{2\}}_n$ for $s=0$ becomes
$\left\{\{i_1\},\{i_2\},\dots,\{i_a\}\right\}$. All other
orbits become $\{\{i_g,i_h\} \mid 1 \leq g,h \leq a, g \equiv h+s
\bmod a\}$ in the same way as before, but these are identical for $s$
and $a-s$. This yields one further exception if $a$ is even and $s=a/2$,
in which case the orbit collapses to $\{\{i_g,i_{g+a/2}\} \mid 1\leq
g\leq a/2\}$. In total, identical cycles lead to the monomials 
\begin{eqnarray*}
x_{a/2}x_a^{a/2} & \mbox{ if } &  a \equiv 0 \mod 2\\
x_a^{(a+1)/2} & \mbox{ if } &  a \equiv 1 \mod 2.
\end{eqnarray*}
Summing once more over conjugacy classes and making the case split
depending on the parity proves the formula for
$\mathcal{Z}(S_n^{\{2\}})$.
\end{proof}

Formulae like those in the previous lemma for slightly different actions
are given in~\cite[(4.1.9)]{HP73} and~\cite[(5.1.5)]{HP73}. The proof
techniques used here are essentially the same as in~\cite{HP73}. 


\section{Proofs of the main theorems}\label{sec_proof}

In this section, we prove Theorem~\ref{thm_up_to_iso}.  The proofs of Theorems
\ref{thm_up_to_equi}, and \ref{thm_comm} are very  similar to the proof of  Theorem~\ref{thm_up_to_iso}, and so, for the sake of brevity we  show how to obtain these proofs from the one presented, rather than giving the proofs in detail.

We consider the following sets of nilpotent semigroups of degree $3$:
for $m,n\in\N$ with $2\leq m\leq n-1$ we define
\begin{equation*}
Z_{n,m}=\left\{\:\U([n]\setminus[m], \psi, 1) \mid 
\psi:[n]\setminus [m]\times [n]\setminus [m]\to [m]
\mbox{ with }[m]\setminus \{1\} \subseteq \im(\psi)\:\right\},
\end{equation*}
where $\U([n]\setminus[m], \psi, 1)$ is as in Definition
\ref{defn_nil}, and $[n]$ is short for $\{1,2,\ldots, n\}$, as
before. From this point on, we will only consider semigroups
belonging to $Z_{n,m}$, and so we write $\U(\psi)$ instead of
$\U([n]\setminus [m], \psi, 1)$.

If $\U(\psi)\in Z_{n,m}$ is commutative, then we define a function $\psi'$ 
from the set of subsets of $[n]$ with $1$ or $2$ elements to $[m]$ by 
\begin{equation}\label{sets}
\psi'\{i,j\}=\psi(i,j)
\end{equation}
for $i \leq j$.
Since the equality $\psi(i,j)=\psi(j,i)$ holds for all $i$, $j$, 
 the function $\psi'$ is well-defined. 
Moreover, every  function from the set of subsets of $[n]$ with $1$ or
$2$ elements to $[m]$ is induced in this way by a function $\psi$ such
that $\U(\psi)\in Z_{n,m}$ and $\U(\psi)$ is commutative. 

\begin{lemma}
\label{lem_unique_non_gens}
Let $S$ be a nilpotent semigroup of degree $3$ with $n$ elements. Then
$S$ is isomorphic to a semigroup in $Z_{n,m}$ if and only if $m=|S^2|$.
\end{lemma}
\begin{proof} 
Let $z$ denote the zero element of $S$, and let  $f:S\to [n]$ be any
bijection such that $f(z)=1$ and $f(S^2)=[m]$. Then define
$\psi:([n]\setminus [m])\times ([n]\setminus [m])\rightarrow [m]$ by
$$\psi(i,j)=f(f^{-1}(i)f^{-1}(j)).$$
Now, since $S$ is nilpotent, if $x\in [m]\setminus \{1\}$, there exist
$s,t\in S\setminus S^2$ such that $f(st)=x$. Thus $\psi(f(s), f(t))=x$ and
$[m]\setminus \{1\} \subseteq \im(\psi)$. Hence $\U(\psi)\in Z_{n,m}$ and it
remains to show that $f$ is an isomorphism. If $x,y\in S\setminus
S^2$, then $f(x)f(y)=\psi(f(x),f(y)) = f(xy)$. 
Otherwise, $x\in S^2$ or $y\in S^2$, in which case
$f(x)f(y)=1=f(z)=f(xy)$.
\end{proof}

It follows from Lemma \ref{lem_unique_non_gens} that we can determine the number of 
isomorphism types in each of the  sets $Z_{n,m}$ independently. 
Of course, if $S$ is a nilpotent semigroup of degree $3$ and $m=|S^2|$, then it is not true in general that there
exists a unique semigroup in $Z_{n,m}$ isomorphic to $S$. 
Instead isomorphisms between semigroups in $Z_{n,m}$ induce an
equivalence relation on the functions  $\psi$, which define the
semigroups in $Z_{n,m}$. 

If $\U(\psi)\in Z_{n,m}$ and $T$ is a nilpotent semigroup of degree $3$
such that $\U(\psi)\cong T$, then there exists $\pi\in
S_n$ such that $S^\pi=T$. Hence $T\in Z_{n,m}$ if and only if $\pi$ stabilises
$[n]\setminus [m]$ and $\{1\}$ --
and hence $[m]\setminus \{1\}$ -- setwise.
In particular, the action of $\pi$ on the domain 
and range of $\psi$ are independent, and so equivalence can be
captured using a power group action.


\begin{lemma}\label{lem_iso}
For $m,n\in \N$ with $2\leq m\leq n-1$ let $\U(\psi), \U(\chi)\in
Z_{n,m}$, and let $U_m$ denote the pointwise stabiliser of $1$ in
$S_m$. Then the following hold: 
\begin{enumerate}
\item  the semigroups $\U(\psi)$ and $\U(\chi)$ are isomorphic if and
  only if $\psi$ and $\chi$ are in the same orbit under the power group
  $S_{[n]\setminus [m]}^{\times 2}\times U_m$;
\item the semigroups $\U(\psi)$ and $\U(\chi)$ are isomorphic or
  anti-isomorphic if and only if $\psi$ and $\chi$ are in the same
  orbit under the power group $2S_{[n]\setminus [m]}^{\times 2}\times U_m$.
\end{enumerate}
If in addition $\U(\psi)$ and $\U(\chi)$ are commutative, then:
\begin{enumerate}
\item[(iii)] the semigroups $\U(\psi)$ and $\U(\chi)$ are
  isomorphic if and only if $\psi'$ and $\chi'$
  (as defined in \eqref{sets}) are in the same orbit under the power group
  $S_{[n]\setminus [m]}^{\{2\}}\times U_m$.
\end{enumerate}
\end{lemma}
\begin{proof} 
{\bf (i).}
($\Rightarrow$) By assumption there exists  $\pi\in S_{n}$
such that $\pi: \U(\psi)\to \U(\chi)$ is an isomorphism. 
From the comments before the lemma, $\pi$ stabilises $[n]\setminus [m]$ and $1$,
and so there exist $\tau\in U_m$ and 
$\sigma\in S_{[n]\setminus [m]}$ such that $\tau\sigma=\pi$. Then for all
$x,y\in [n] \setminus [m]$
\[
\psi(x,y)=(\psi(x,y)^{\pi})^{\pi^{-1}} = (\chi(x^{\pi},
y^{\pi}))^{\pi^{-1}} = (\chi(x^{\sigma}, y^{\sigma}))^{\tau^{-1}}.
\]
It follows that $\chi$ acted on by $(\sigma, \tau^{-1})\in
S_{[n]\setminus [m]}^{\times 2}\times U_m$ equals $\psi$, as required. 

($\Leftarrow$) Since $\psi$ and $\chi$  lie in the same orbit under
the action of the power group $S_{[n]\setminus [m]}^{\times 2}\times U_m$, there exist
$\sigma\in S_{[n]\setminus [m]}$ and $\tau\in U_m$ such that $\psi^{(\sigma, \tau)}=\chi$. 
Let $\pi=\sigma\tau^{-1} \in S_{n}$.
We will show that $\pi$ is an isomorphism from $\U(\psi)$ to
$\U(\chi)$. Let $x,y\in [n]$ be arbitrary. If $x,y\in [n]\setminus [m]$, then 
\[
x^{\pi}y^{\pi} = \psi(x^{\sigma}, y^{\sigma}) =
(\psi(x^{\sigma},y^{\sigma})^{\tau})^{\tau^{-1}} = 
(\psi^{(\sigma, \tau)}(x,y))^{\tau^{-1}} = (\chi(x,y))^{\tau^{-1}} =
(xy)^{\pi}.
\]
If $x\in [n]\setminus[m]$ and $y\in[m]$, then 
$(xy)^{\pi}=1^{\pi}=1=x^{\sigma} y^{\tau^{-1}} = x^{\pi}y^{\pi}$.
The case when $x\in [m]$ and $y\in[n]\setminus[m]$ and the case when $x,y\in
[m]$ follow by similar  arguments. 

{\bf (ii).} In this part of the proof we write $(\alpha, \beta, \gamma)$ instead of $((\alpha, \beta), \gamma)$ when referring to elements of $2S_{[n]\setminus [m]}^{\times 2}\times U_m$.

($\Rightarrow$) If $\U(\psi)$ and $\U(\chi)$ are isomorphic, then, by part (i), the functions $\psi$ and $\chi$ are in the same orbit under the
action of $S_{[n]\setminus [m]}^{\times 2}\times U_m$.
Since $S_{[n]\setminus [m]}^{\times 2}\times U_m$ is contained in $2S_{[n]\setminus [m]}^{\times 2}\times U_m$, 
it follows that $\psi$ and $\chi$ are in the same orbit
under the action of $2S_{[n]\setminus [m]}^{\times 2}\times U_m$.

If $\U(\psi)$ and $\U(\chi)$ are not isomorphic, then there exists  $\pi \in S_{n}$ such that 
$\pi:\U(\psi)\to \U(\chi)$ is an anti-isomorphism. As in the
proof of part (i), there exist $\tau\in U_m$ and $\sigma\in S_{[n]\setminus [m]}$ such that $\pi=\tau\sigma$. 
Then, for all $x,y\in [n]\setminus [m]$,
\[
\psi(x,y)=(\psi(x,y)^{\pi})^{\pi^{-1}} = (\chi(y^{\pi},
x^{\pi}))^{\pi^{-1}} = (\chi(y^{\sigma}, x^{\sigma}))^{\tau^{-1}} =
\chi^{(\sigma, \tau^{-1})}(y,x).
\]
Hence $\chi$ acted on by $((1\,2), \sigma, \tau^{-1})\in 2S_{[n]\setminus [m]}^{\times 2}\times U_m$ equals $\psi$. 

($\Leftarrow$) If $\psi=\chi^{(\id_{\{1,2\}}, \sigma, \tau)}$ for some $(\id_{\{1,2\}}, \sigma, \tau)\in 2S_{[n]\setminus [m]}^{\times 2}\times U_m$, 
then $\U(\psi)$ and $\U(\chi)$ are isomorphic by part (i). 
So, we may assume that
$\psi=\chi^{((1\,2),\sigma, \tau)}$. Let $\pi=\sigma\tau^{-1} \in S_{n}$.
We show that $\pi$ is an anti-isomorphism from $\U(\psi)$ to
$\U(\chi)$. Let $x,y\in [n]$ be arbitrary. If $x,y\in [n]\setminus [m]$, then 
\[
x^{\pi}y^{\pi} = \psi(x^{\sigma}, y^{\sigma}) =
(\psi(x^{\sigma},y^{\sigma})^{\tau})^{\tau^{-1}} = 
(\psi^{((1\,2),\sigma, \tau)}(y,x))^{\tau^{-1}} = (\chi(y,x))^{\tau^{-1}} =
(yx)^{\pi}.
\]
If $x\in [n]\setminus[m]$ and $y\in[m]$, then 
$(xy)^{\pi}=1^{\pi}=1=y^{\tau^{-1}} x^{\sigma} = y^{\pi}x^{\pi}$.
The case when $x\in [m]$ and $y\in[n]\setminus[m]$ and the case when $x,y\in
[m]$ follow by similar  arguments. 

{\bf (iii).} 
The proof follows from (i) and the observation that $\psi'$ and $\chi'$
are in the same orbit under $S_{[n]\setminus [m]}^{\{2\}}\times U_m$ if and only if
$\psi$ and $\chi$ are in the same orbit under $S_{[n]\setminus [m]}^{\times 2}\times U_m$.      
\end{proof}


Lemma~\ref{lem_iso}(i) shows that the number of non-isomorphic semigroups in
$Z_{n,m}$ equals the number of orbits of functions defining
semigroups in $Z_{n,m}$ under the appropriate power group action. 
Together with Theorem~\ref{thm_PGET} this provides the essential
information required to prove the formula given in
Theorem~\ref{thm_up_to_iso} for the number of
nilpotent semigroups of degree $3$ of order $n$ up to isomorphism.

\noindent\proofref{thm_up_to_iso}
Denote by $U_q$ the stabiliser of $1$ in $S_q$. We shall first show
that $N(p,q)$ is the number of orbits of the power
group $S_{[p]\setminus [q]}^{\times 2}\times U_q$ on functions from
$([p]\setminus [q]) \times ([p]\setminus [q])$ to $[q]$. By
Theorem~\ref{thm_PGET} the latter equals
\begin{equation}
\label{Neq}
\frac{1}{(q-1)!}\sum_{\beta\in H}\mathcal{Z}(S_{[p]\setminus
  [q]}^{\times 2}; c_1(\beta), \ldots, c_{(p-q)^2}(\beta)),
\end{equation}
where 
\[
c_{i}(\beta)= \sum_{d \mid i} d \delta(\beta,d).
\]
If $\beta\in U_q$, then $\mathcal{Z}(S_{[p]\setminus [q]}^{\times 2};
c_1(\beta), \ldots, c_{(p-q)^2}(\beta))$ only depends on the cycle
structure of $\beta$ and is therefore an invariant of the conjugacy classes of
$U_q$. These conjugacy classes are in 1-1 correspondence with the
partitions of $q-1$. If $j$ is a partition of $q-1$ corresponding to
the conjugacy class of $\beta$, then $\delta(\beta,1) = j_1 + 1$ and
$\delta(\beta,i) = j_i$ for $i = 2,\dots, q-1$ (where $j_i$ denotes,
as before, the number of summands in $j$ equalling $i$). This yields that
$c_i(\beta)= 1 + \sum_{d \mid i} d\,j_d$. 
The size of the conjugacy class in $U_q$ corresponding to the partition $j$ is
$(q-1)!/\prod_{i=1}^{q-1} j_i!\,i^{j_i}$.
Hence summing over conjugacy classes in \eqref{Neq} gives:
\begin{equation}
\label{Nwithc}
\sum_{j \vdash q-1} \left(\prod_{i=1}^{q-1}
  j_i!\,i^{j_i}\right)^{-1}
\mathcal{Z}\left(S^{\times 2}_{[p]\setminus [q]}; 1+\sum_{d \mid 1} d j_d,\ldots, 1+\sum_{d \mid (p-q)^2} d j_d  \right).
\end{equation}
Substituting the cycle index of $S_{[p]\setminus [q]}^{\times 2}$ from
Lemma \ref{lem_cyc_ind}(i) into \eqref{Nwithc} yields the formula
given in the statement of the Theorem for $N(p,q)$.

By Lemma~\ref{lem_iso}(i), the number of non-isomorphic semigroups in
$Z_{n,m}$ for $m \in \N$ with $2\leq m\leq n-1$ equals the number of
orbits under the power group $S_{[n]\setminus [m]}^{\times 2}\times
U_m$ of functions from $([n]\setminus [m])\times ([n]\setminus [m])$ to
$[m]$ having $[m]\setminus \{1\}$ in their image. 
The orbits counted in $N(n,m)$ include those of functions
which do not contain $[m]\setminus \{1\}$ in their image. The number
of such orbits equals $N(n-1,m-1)$, the number of orbits of functions
with one fewer element in the image set. Hence the number of
non-isomorphic semigroups in $Z_{n,m}$ equals
$N(n,m)-N(n-1,m-1)$. With Lemma~\ref{lem_unique_non_gens}, it
follows that the number of non-isomorphic nilpotent semigroups of degree
$3$ with $n$ elements is
\begin{equation*}
\sum_{m=2}^{a(n)}\left(N(m,n)-N(m-1, n-1)\right)\mbox{ \ where \ } 
a(n)=\left\lfloor n+1/2-\sqrt{n-3/4}\,\right\rfloor.\qedhere
\end{equation*}

Replacing the cycle index in \eqref{Neq} by that of 
$2S^{\times 2}_{[p]\setminus [q]}$ and $S^{\{2\}}_{[p]\setminus [q]}$
proves Theorems~\ref{thm_up_to_equi} and \ref{thm_comm}, respectively, using the same argument as above.


\bibliographystyle{alpha}
\bibliography{3nil}{}

\end{document}